\documentclass[12pt,reqno]{amsart}
\usepackage{amssymb}
\usepackage{amsmath}
\usepackage{amsfonts}
\usepackage{color}

\newtheorem{theorem}{Theorem}[section]
\newtheorem{corollary}[theorem]{Corollary}
\newtheorem{notation}[theorem]{Notation}
\newtheorem{lemma}[theorem]{Lemma}
\newtheorem{proposition}[theorem]{Proposition}

\theoremstyle{remark}
\newtheorem{remark}[theorem]{Remark}

\numberwithin{equation}{section}

\begin{document}

\title{Flat connections and cohomology invariants}

\author[I. Biswas]{Indranil Biswas}

\address{School of Mathematics, Tata Institute of Fundamental
Research, Homi Bhabha Road, Bombay 400005, India}

\email{indranil@math.tifr.res.in}

\author[M. Castrill\'{o}n L\'{o}pez]{Marco Castrill\'{o}n L\'{o}pez}

\address{ICMAT(CSIC-UAM-UC3M-UCM), Dept. Geometr\'{\i}a y
Topolog\'{\i}a, Facultad de Ciencias Matem\'{a}ticas, Universidad
Complutense de Madrid, 28040, Madrid, Spain}

\email{mcastri@mat.ucm.es}



\date{}

\begin{abstract}
The main goal of this article is to construct some
geometric invariants for the topology of the set
$\mathcal{F}$ of flat connections on a principal $G$-bundle $P\,\longrightarrow\,
M$. Although the characteristic classes of principal bundles are
trivial when $\mathcal{F}\neq \emptyset$, their classical
Chern-Weil construction can still be exploited to define a
homomorphism from the set of homology classes of maps $S\,\longrightarrow\,
\mathcal{F}$ to the cohomology group $H^{2r-k}(M,\,\mathbb{R})$, where
$S$ is null-cobordant $(k-1)$-manifold, once a $G$-invariant
polynomial $p$ of degree $r$ on $\text{Lie}(G)$ is fixed. For $S\,=\,S^{k-1}$, this gives a
homomorphism $\pi _{k-1}(\mathcal{F})\,\longrightarrow\, H^{2r-k}(M,\, \mathbb{R})$. The
map is shown to be globally gauge invariant and furthermore it
descends to the moduli space of flat connections
$\mathcal{F}/\mathrm{Gau}P$, modulo cohomology with integer coefficients. The
construction is also adapted to complex manifolds. In this case,
one works with the set $\mathcal{F}^{0,2}$ of connections with
vanishing $(0,2)$-part of the curvature, and the Dolbeault
cohomology. Some examples and applications are presented.
\end{abstract}

\maketitle

\section{Introduction}

The classical Chern-Weil construction of the characteristic
classes of a principal $G$-bundle $\pi\,:\, P \,\longrightarrow \,
M$ on a manifold $M$ of dimension $n$ can be built in terms of the
bundle of connections $C\,\longrightarrow \, M$ associated to
$\pi$, that is, the affine bundle over $M$ whose sections over any open
subset $U\subset M$ are the connection on the restriction
$P\vert_U$. The bundle $\pi^* P \,\longrightarrow \, C$ is endowed
with a canonical connection so that the evaluation of its
curvature $\Omega$ by the $G$-invariant polynomials of
$\mathfrak{g}\,=\,\text{Lie}(G)$ provides a sort of universal
characteristic form $p(\Omega )$ inducing the characteristic
classes under the isomorphism $H^{\bullet}(M)\,\simeq\,
H^{\bullet}(C)$. One of the advantages of this approach relies on
the fact that the dimension $(1+\dim G)n$ of $C$ is bigger than the dimension
$n$ of $M$, so that polynomials $p$ of degree $r$ bigger than
$n/2$ still provide forms of potential relevance. For example, in
\cite{F}, these polynomials are used to define a form of degree
$2r-n$ on the space of all connections $\mathcal{A}$. In this
article, we construct forms $\beta ^p _k$ of degree $2r-k$, for
any $0\leq k\leq r$.

The topological information provided by characteristic classes is
trivial if the bundle $\pi$ admits flat connections. However,
there are many interesting instances of these special bundles,
both in physics and in mathematics: bundles modelling quantum
phases as that of the Aharonov-Bohm effect, topological Field
Theories, moduli spaces of Yang-Mills solution on Riemann
surfaces, stability of vector bundles over algebraic varieties,
etcetera. Interestingly, we show that the above mentioned forms
$\beta ^p _k$ provide topological information by integrating along
submanifolds contained in the subset
$\mathcal{F}\,\subset\,\mathcal{A}$ of flat connections when
$\mathcal{F}\,\neq\, \emptyset$. More precisely, we construct a map
$$\mathrm{Map}(S,\,\mathcal{F})\,\longrightarrow\,
H^{2r-k}(M,\,\mathbb{R})$$ for any manifold $S$ of dimension $k-1$
which is trivially cobordant. Furthermore, it is proved that the
above map behaves well with respect to homology relation in
$\mathrm{Map}(S,\,\mathcal{F})$ and also it is invariant under the
gauge group $\mathrm{Gau}P$ acting on the left on
$\mathrm{Map}(S,\,\mathcal{F})$. These results are connected with
previous articles in the literature (see \cite{Gu} and \cite{Gu2})
where similar constructions were obtained, but from a different
approach and dealing with homology chains in $\mathcal{F}$ and for
$S=S^1$ respectively. The case of spheres $S=S^{k-1}$ is of
special interest as it induces a homomorphism defined on higher homotopy
groups $\pi _{k-1} (\mathcal{F})\,\longrightarrow \,
H^{2r-k}(M,\,\mathbb{R})$.

The invariance under gauge transformation is not enough to
guarantee a well defined object in the moduli space of flat
connections $\mathcal{F}/\mathrm{Gau}P$. In fact, the moduli space
involves actions of $\Phi _s \in \mathrm{Gau}P$ in $\mathcal{F}$
depending on $s\in S$ itself, a situation for which the invariance
above does not hold true. Still, we prove that the our main
construction descends to a map defined on the set
\[
\{\Pi\circ f \,\in\, \mathrm{Map}(S, \,\mathcal{F}/\mathrm{Gau}P)
\,\mid\, f\,\in\, \mathrm{Map}(S,\,\mathcal{F})\}\, ,
\]
where $\Pi: \mathcal{A}\longrightarrow\mathcal{A}/\mathrm{Gau}P$, and taking values in
$H^{2r-k}(M,\,\mathbb{R})$ modulo entire cohomology $H^{2r-k}(M,\,
\mathbb{Z})$. This construction is consistent with particular cases previously
given in the literature (see \cite{Gu2} and \cite{Sal}).

We also extend the study of these geometric invariants to complex
bundles and manifolds. In particular, suitable adaptations of the
main objects allow a definition of a map taking values in the
Dolbeault cohomology $H^{0,2r-k}(M)$, from homotopy classes of
maps in the space of connections with vanishing $(0,2)$-part of
the curvature. We remark that this type of connections is of
essential interest when studying holomorphic principal
connections, so that a bridge between the problem of existence of
them with the geometric objects constructed by invariant
polynomials is presented.

Roughly, the contents of the paper are organized as follows.
Section 2 builds the basic ingredients used in Section 3 to define
the invariants induced by maps to the space of flat connections
$\mathcal{F}$. Section 4 tackles the case of complex manifolds and
the formulation of the geometric invariants in the language of
Dolbeault cohomology. Finally, Section 5 is devoted to some
particular examples and applications which may serve as a
motivation for future further investigation of the invariants
herein defined.

All the objects in this article are smooth and the Einstein
notation on repeated indices is assumed.

\section{Forms in the space of connections $\mathcal{A}$}

Let $M$ be a $C^\infty$ manifold of dimension $n$, and let $G$ be
a Lie group of dimension $m$. The Lie algebra of $G$ will be
denoted by $\mathfrak g$. Let $\pi\,:\,P\,\longrightarrow\, M$ be
a $C^\infty$ principal $G$-bundle with
$$
d\pi\,:\, TP\,\longrightarrow\,\pi^*TM
$$
being the differential of $\pi$. For any $x\, \in\, M$, the fiber
$\pi^{-1}(x)$ will be denoted by $P_x$. The vector bundle
$$\pi_{\mathrm{ad}}\,:\, \mathrm{ad}P\, :=\, P\times^G {\mathfrak g}
\,\longrightarrow\, M$$ associated to $P$ for the adjoint action
of $G$ on $\mathfrak g$ is known as the adjoint vector bundle of $P$. We
note that $\mathrm{ad}P\,=\, (\text{kernel}(d\pi))/G$. Therefore,
any $\xi\,\in\, (\mathrm{ad}P)_{x}$, $x\,\in\, M$, produces a
vector field along the fiber $\pi^{-1}(x)$ which is preserved by
the action of $G$ on $\pi^{-1}(x)$.

Let $\text{At}(P)\,:=\, (TP)/G\,\longrightarrow\, M$ be the Atiyah bundle
for $P$. It fits in a short exact sequence
$$
0\,\longrightarrow\, \mathrm{ad}P \,\longrightarrow\,
\text{At}(P) \,\stackrel{d\pi}{\longrightarrow}\, TM
\,\longrightarrow\, 0
$$
which is known as the Atiyah exact sequence (see \cite{Ati}).
Consider
$$
\text{At}(P)\otimes T^*M \,\stackrel{d\pi\otimes {\rm
Id}}{\longrightarrow}\, TM\otimes T^*M \,=\, {\rm End}(TM)\, ,
$$
and define $C\,:=\, (d\pi\otimes {\rm Id})^{-1}({\rm Id}_{TM})\,\subset\,
\text{At}(P)\otimes T^*M$. Clearly,
\begin{equation}\label{eqq}
q\,:\, C\,\longrightarrow\,M
\end{equation}
is an affine bundle for the vector bundle $(\mathrm{ad}P)\otimes T^*M$. This
affine bundle $(C\, ,q)$ is known as the bundle of connections for $P$. Any
section $A\,:\,M\,\longrightarrow\, C$ of $q$ is thus a connection on
$P$, that is
\begin{equation}\label{e1}
\mathcal{A}\,=\,\Gamma(M,\, C)\, ,
\end{equation}
where $\mathcal{A}$ is the space of all connections on $P$. The
difference $A_{1}-A_{2}$ of two connections $A_{1},A_{2}\,\in\,\mathcal{A}$ is a
$1$-form in $M$ taking values in $\mathrm{ad}P$; note that the
kernel of $d\pi\otimes {\rm Id}$ is $\mathrm{ad}P\otimes T^*M$. The isomorphism in
\eqref{e1} is compatible with the affine space structures of $C$ and $\mathcal{A}$.

Any element of $q^{-1}(x)\, \subset\, C$, $x\, \in\, M$, provides
a $G$-invariant homomorphism of vector
bundles $P_x\times T_xM\, \longrightarrow\, (TP)\vert_{P_x}$ over $P_x$
whose pre-composition with $d\pi$ is the
identity map of $P_x\times T_xM$; it is a horizontal lift of tangent vectors over $x$.
We now consider a coordinate domain $U\,\subset\, M$, with
coordinates $(x^{1},\cdots ,x^{n})$, and such that $\pi^{-1}(U)
\,\simeq\, U\times G$ (after choosing a trivialization of
$P\vert_U$). For any $B\,\in\,\mathfrak{g}$, let
$\widetilde{B}$ be the vertical $G$-invariant vector field on
$P\vert_U\,\simeq\, U\times G$ defined as
\[
\widetilde{B}_{(x,g)}\,=\,\left. \frac{d}{dt}\right\vert _{t=0}(x,\exp(tB)g)\, .
\]
If $(B_{1}\, ,\cdots\, ,B_{m})$ is a basis of $\mathfrak{g}$, the
system $(\widetilde{B}_{1},\cdots ,\widetilde{B}_{m})$ is a basis
of sections of $\mathrm{ad}P$ considered as a
$C^\infty(U)$-module. In addition, the horizontal lift given by
$A_x\,\in\, q^{-1}(x)$ has an expression like
$$
\frac{\partial}{\partial x^i}\,\longmapsto\, \frac{\partial}{\partial x^i}
- A^{\alpha} _i \widetilde{B}_{\alpha}\, .
$$
The functions $(x^i,A^\alpha _j)$, $i,j\,=\,1,\cdots ,n$, $\alpha
\,=\,1,\cdots , m$, define a coordinate system on $q^{-1}(U)\,\subset\,
C$.

The bundle of connections is equipped with a canonical $2$-form
$F$ taking values in the bundle
$q^{\ast}\mathrm{ad}P\,\longrightarrow\, C$ called the universal
curvature. It is the curvature form of a canonical connection
defined on the principal $G$-bundle $q^{\ast}P\,\longrightarrow\,
C$ (for example, see \cite{G} as well as \cite{B}, \cite{BHS}, \cite{CM}) and
satisfies the following condition: given a section $A\,\in\, \Gamma(C)$ of the bundle
of connections, the pulled back form
\[
A^{\ast}F\,\in\, \Omega^{2}(M,\, A^*q^{\ast}\mathrm{ad}P)
\,=\,\Omega^{2}(M,\, \mathrm{ad}P)
\]
coincides with the curvature $F^{A}$ of the connection $A$. For a coordinate
system on $C$ as described above, the expression of the form $F$ is
\begin{equation}
F\,=\,(dA_{i}^{\alpha}\wedge
dx^{i}+\tfrac{1}{2}c_{\beta\gamma}^{\alpha}A_{i}^{\beta}
A_{j}^{\gamma}dx^{i}dx^{j})\widetilde{B}_{\alpha}\, ,\label{unicu}
\end{equation}
where $c_{\beta\gamma}^{\alpha}$ are the structure constants of the basis
$(B_{1},\cdots ,B_{m})$ of $\mathfrak{g}$.

Let $p\,\in\, S^{r}(\mathfrak{g}^{\ast})^{G}$ be a symmetric polynomial of degree $r$
on $\mathfrak{g}$ which
is invariant under the adjoint action of $G$ on $\mathfrak g$. We
define a $2r$-form in $C$ in the usual way:
\[
p(F)\,=\, p(F,\overset{(r)}{\cdots},F)\,\in\,\Omega^{2r}(C)\, .
\]
{}From the Chern-Weil theory we know that the form $p(F)$ is closed. For any
$A\,\in\,\Gamma(C)$, the pull-back
$A^{\ast}p(F)\,\in\,\Omega^{2r}(M)$ is the characteristic form
$p(F^{A})$ defined by the Chern-Weil theory. Note that the
homomorphism $q^*\,:\, H^{\bullet}(M,\, \mathbb{R})\, \longrightarrow\,
H^{\bullet }(C,\, \mathbb{R})$ is an isomorphism. The forms $p(F)$
produce the characteristic classes of the principal $G$-bundle $P$.

A gauge transformation of $P$ is a $G$-equivariant diffeomorphism
$\Phi\,:\,P\,\longrightarrow\, P$ such that $\pi\circ\Phi\,=\,\pi$. The set
of gauge transformations of $P$, which will be
denoted by $\mathrm{Gau}P$, is a group under the composition of
maps. This group $\mathrm{Gau}P$ acts on connections so that any
$\Phi \,\in\, \mathrm{Gau}P$ defines a transformation $\Phi
_{\mathcal{A}}\,:\, \mathcal{A}\, \longrightarrow\, \mathcal{A}$. At the level the
bundle of connections, this action also induces an affine
isomorphism
\begin{equation}\label{pc}
\Phi_{C}\,:\,C\,\longrightarrow\, C \, .
\end{equation}
It is easy to see that the characteristic forms $p(F)$ defined
above are gauge invariant, that is, $\Phi _C^* p(F)\,=\,p(F)$, for all
$\Phi \,\in \,\mathrm{Gau}P$. It is known that the algebra of gauge
invariant forms in $C$ is precisely the ring
\[
\Omega^{\bullet}(C)^{\mathrm{Gau}P}\,=\,\Omega^{\bullet}(M)[p_{1}(F),\cdots ,p_{l}
(F)]\, ,
\]
where $p_{1},\cdots ,p_{l}$ generate $S^{\bullet}(\mathfrak{g}^{\ast
})^{G}$ (see \cite{CM2}).

For any $p\,\in\, S^{r}(\mathfrak{g}^{\ast})^{G}$ and
$0\,\leq\, k\,\leq\,2r$, we construct a $k$-form
\[
\beta_{k}^{p}\,\in\,\Omega^{k}(\mathcal{A},\, \Omega^{2r-k}(M))
\]
on $\mathcal{A}$ (defined in \eqref{e1})
taking values in the vector space $\Omega^{2r-k}(M)$ as follows:
\begin{equation}
\beta_{k}^{p}(A)(\xi_{1},\cdots ,\xi_{k})\,=\,A^{\ast}(i_{\xi_{1}}\cdots i_{\xi_{k}
}p(F))\, ,\label{beta}
\end{equation}
for $A\,\in\,\mathcal{A}\,=\,\Gamma(C)$, and $\xi_{1},\cdots
,\xi_{k}\,\in\, T_{A}\mathcal{A}\,=\, \Omega^{1}(M,\mathrm{ad}P)$; by
$i_{\xi_{j}}$ we denote the contraction of differential forms by the vector
field $\xi_{j}$.
As $C\,\longrightarrow\, M$ is an affine bundle modelled on the vector bundle
$T^{\ast}M\otimes\mathrm{ad}P\,\longrightarrow\, M$, we consider
$\xi_{i}$ as a vertical vector field, with respect to the
projection $q$, along $A(M)$. If $k\,=\,0$, then
$\beta_{0}^{p}\,\in\, C^{\infty }(\mathcal{A},\,\Omega^{2r}(M))$ is
just the characteristic form $\beta_{0}^{p}(A)\,=\,p(F^{A})$.

We assume in the following that $k\,>\,0$.

\begin{lemma}
\label{lem2.1}
 For $k\,\leq\, r$, we have
\[
\beta_{k}^{p}(A)(\xi_{1},\cdots ,\xi_{k})\,=\,N_{r,k}p(\xi_{1},\cdots ,\xi_{k},F^{A}
,\overset{(r-k)}{\cdots},F^{A})\, ,
\]
with $N_{r,k}\,=\,r(r-1)\cdots(r-k+1)\,=\, \frac{r!}{(r-k)!}$, where
$\xi_{1},\cdots,\xi_{k}\,\in\,
T_{A}\mathcal{A}\,=\,\Omega^{1}(M,\,\mathrm{ad}P)$.

If $r\,<\,k\,\leq\,2k$, then we have $\beta_{k}^{p}\,=\,0$.
\end{lemma}

\begin{proof}
Using formula (\ref{unicu}), it is easy to see that for $\xi_{1}\, ,\xi_{2}
\,\in\,\Omega^{1}(M,\mathrm{ad}P)$, we have
\[
i_{\xi_{1}}F\,=\,q^{\ast}\xi_{1}\,,\qquad i_{\xi_{2}}i_{\xi_{1}}F\,=\,0\, ,
\]
where $q$ is the projection in \eqref{eqq}. Therefore, we have
\[
i_{\xi}p(F)\,=\,r\cdot p(q^{\ast}\xi,F,\overset{(r-1)}{\cdots},F)
\]
and then, for $k\,\leq \,r$,
\begin{align*}
\beta_{k}^{p}(A)(\xi_{1},\cdots ,\xi_{k}) & =\,A^{\ast}(i_{\xi_{1}}\cdots i_{\xi_{k}
}p(F))\\
& =\,A^{\ast}(N_{r,k}p(q^{\ast}\xi_{1},\cdots ,q^{\ast}\xi_{k},F,\overset{(r-k)}
{\cdots},F)\\
& =\,N_{r,k}p(\xi_{1},\cdots ,\xi_{k},F^{A},\overset{(r-k)}{\cdots},F^{A}).
\end{align*}
For $k\,>\,r$, we have $\beta_{k}^{p}(\xi_{1},\cdots ,\xi_{k})\,=\,0$.
\end{proof}

\begin{proposition}
\label{inv}The form $\beta_{k}^{p}$ is invariant under the action of the gauge
group in $\mathcal{A}$, in other words,
\[
\Phi_{\mathcal{A}}^{\ast}(\beta_{k}^{p})\,=\,\beta_{k}^{p},\qquad\forall\ \Phi
\,\in\, \mathrm{Gau}P\, ,
\]
where $\Phi_{\mathcal{A}}\,:\,\mathcal{A}\,\longrightarrow\, \mathcal{A}$ is the
automorphism induced by $\Phi_{C}$ in \eqref{pc}.
\end{proposition}

\begin{proof}
For any choices of $\Phi\,\in\,\mathrm{Gau}P$ and $\xi_{1},\cdots ,\xi_{k}\,\in\,
T_{A}\mathcal{A}$, we have
\begin{align*}
(\Phi_{\mathcal{A}}^{\ast}\beta_{k}^{p})(A)(\xi_{1},\cdots ,\xi_{k}) & =\,
\beta_{k}^{p}(\Phi_{\mathcal{A}}(A))((\Phi_{\mathcal{A}})_{\ast}(\xi_{1}),\cdots ,(\Phi_{\mathcal{A}})_{\ast
}(\xi_{k}))\\
& =\,\beta_{k}^{p}(\Phi_{\mathcal{A}}(A))(\Phi_{\mathrm{ad}}\circ\xi_{1},\cdots ,
\Phi_{\mathrm{ad}}\circ\xi_{k})\, ,
\end{align*}
where $\Phi_{\mathrm{ad}}\,:\,\mathrm{ad}P\,\longrightarrow\,\mathrm{ad}P$ is the
automorphism of the adjoint bundle induced by $$\Phi\,:\,P\,\longrightarrow\, P\, .$$
Then
\begin{align*}
(\Phi_{\mathcal{A}}^{\ast}\beta_{k}^{p})(A)(\xi_{1},\cdots ,\xi_{k}) &
=\,N_{r,k}p(\Phi_{\mathrm{ad}}\circ\xi_{1},\cdots ,\Phi_{\mathrm{ad}}\circ\xi_{k},
F^{\Phi_{\mathcal{A}}(A)},\overset{(r-k)}{\cdots},F^{\Phi_{\mathcal{A}}
(A)})\\
& =\, N_{r,k}p(\Phi_{\mathrm{ad}}\circ\xi_{1},\cdots ,\Phi_{\mathrm{ad}}\circ\xi_{k}
,\Phi_{\mathrm{ad}}\circ F^{A},\overset{(r-k)}{\cdots},\Phi_{\mathrm{ad}}\circ
F^{A})\\
& =\,N_{r,k}p(\xi_{1},\cdots ,\xi_{k},F^{A},\overset{(r-k)}{\cdots},F^{A})\\
& =\,\beta_{k}^{p}(\xi_{1},\cdots,\xi_{k})\, ,
\end{align*}
by taking into account the invariance of $p$ under the adjoint action.
\end{proof}

\section{Cohomology defined by maps to $\mathcal{F\subset A}$\label{secF}}

\subsection{Definition of the invariant}

\begin{proposition}
\label{Prop1}
\emph{The identity}
\[
d_{\mathcal{A}}\beta_{k}^{p}\,=\,(r-k)~d\circ\beta_{k+1}^{p}
\]
holds for all $0\,\leq\, k\,\leq\, r$, where
\[
d_{\mathcal{A}}\,:\,\Omega^{k}(\mathcal{A},\,\Omega^{2r-k}(M))\,\longrightarrow\,\Omega
^{k+1}(\mathcal{A},\,\Omega^{2r-k}(M))
\]
is the de Rham differential of forms in $\mathcal{A}$ taking values in the vector
space $\Omega^{2r-k}(M)$, and
\[
d\,:\,\Omega^{2r-k-1}(M)\,\longrightarrow\,\Omega^{2r-k}(M)
\]
is the standard de Rham differential of forms in $M$.

In particular, the form $d_{\mathcal{A}}\beta_{k}^{p}$ takes
values in the subspace of exact forms $B^{2r-k}(M)\,:=\,
d\Omega^{2r-k-1}(M) \,\subset\,\Omega^{2r-k}(M)$.
\end{proposition}

\begin{proof}
We have
\begin{gather*}
d_{\mathcal{A}}\beta_{k}(\xi_{0},\cdots ,\xi_{k})\,=\,
{\textstyle\sum}
(-1)^{i}\xi_{i}(\beta_{k}(\xi_{0},\cdots ,\widehat{\xi}_{i},\cdots ,\xi_{k+1}))\\
+
{\textstyle\sum_{i<j}}
(-1)^{i+j}\beta_{k}([\xi_{i},\xi_{j}],\xi_{0},\cdots ,\widehat{\xi}_{i},\cdots ,
\widehat{\xi }_{j},\cdots ,\xi_{k})\\
=\,
{\textstyle\sum}
(-1)^{i}\xi_{i}(A^{\ast}(i_{\xi_{0}}\cdots\widehat{\imath}_{\xi_{i}}\cdots
i_{\xi_{k}}p(F)))\, ,
\end{gather*}
for $\xi_{0},\cdots ,\xi_{k}\,\in\, T_{A}\mathcal{A}$, where $[\xi_{i}\, ,\xi_{j}]\,=\,0$ as
we take $\xi_{i}$ to be the constant vector fields in the affine space
$\mathcal{A}$; by $\widehat{\xi}_{i}$ it is meant that ${\xi}_{i}$ is omitted. Hence
\begin{gather*}
\xi_{i}(A^{\ast}(i_{\xi_{0}}\cdots\widehat{\imath}_{\xi_{i}}\cdots i_{\xi_{k}
}p(F)))\\
=\,N_{r,k}\left. \frac{d}{dt}\right\vert _{t=0}(A+t\xi_{i})^{\ast}p(q^{\ast}
\xi_{0},\cdots ,q^{\ast}\xi_{i-1},q^{\ast}\xi_{i+1},\cdots ,q^{\ast}\xi_{k}
,F,\overset{(r-k)}{\cdots},F)\\
=\,N_{r,k}\left. \frac{d}{dt}\right\vert _{t=0}p(\xi_{0},\cdots ,\xi_{i-1}
,\xi_{i+1},\cdots ,\xi_{k},F^{A+t\xi_{i}},\overset{(r-k)}{\cdots },F^{A+t\xi_{i}})\\
=\,(r-k)N_{r-k}\cdot p(\xi_{0},\cdots ,\xi_{i-1},\xi_{i+1},\cdots ,\xi_{k},\nabla^{A}
\xi_{i},F^{A},\overset{(r-k-1)}{\cdots },F^{A})\, ;
\end{gather*}
we have used the fact that
\[
\left. \frac{d}{dt}\right\vert _{t=0}F^{A+t\xi}\,=\, \nabla^{A}\xi
\]
for $\xi\,\in\,\Omega^{1}(M,\,\mathrm{ad}P)$ and $A\,\in\,\mathcal{A}$. Then
\[
d_{\mathcal{A}}\beta_{k}(\xi_{0},\cdots ,\xi_{k})\,=\,(r-k)N_{r,k}
{\textstyle\sum} (-1)^{i}p(\xi_{0},\cdots
,\xi_{i-1},\nabla^{A}\xi_{i},\xi_{i+1},\cdots ,\xi_{k}
,F^{A},\overset{(r-k-1)}{\cdots},F^{A})\, .
\]

Recall that $C(P)$ is endowed with a derivation law
\[
\nabla\,:\,\Omega^{s}(C(P),\,q^{\ast}\mathrm{ad}P)\,\longrightarrow\,\Omega^{s+1}
(C(P),\, q^{\ast}\mathrm{ad}P)
\]
given by the universal connection of the principal $G$-bundle $q^{\ast}P\,
\longrightarrow\, C$. This derivation satisfies
\begin{gather*}
dp(\xi_{0},\cdots ,\xi_{k},F,\overset{(r-k-1)}{\cdots },F)\,=\,
{\textstyle\sum}
(-1)^{i}p(\xi_{0},\cdots ,\xi_{i-1},\nabla\xi_{i},\xi_{i+1},\cdots ,\xi_{k}
,F,\overset{(r-k-1)}{\cdots },F)\\
+(-1)^{2k}[p(\xi_{0},\cdots ,\xi_{k},\nabla F,\overset{(r-k-1)}{\cdots },F)+\cdots
+p(\xi_{0},\cdots ,\xi_{k},F,\overset{(r-k-1)}{\cdots },\nabla F)]\\
=\,
{\textstyle\sum}
(-1)^{i}p(\xi_{0},\cdots ,\xi_{i-1},\nabla\xi_{i},\xi_{i+1},\cdots ,\xi_{k}
,F,\overset{(r-k-1)}{\cdots },F)\, ,
\end{gather*}
where the last step is obtained by applying the Bianchi identity
$\nabla F\,=\,0$. Hence
\begin{gather*}
dA^{\ast}(p(\xi_{0},\cdots ,\xi_{k},F,\overset{(r-k-1)}{\cdots },F))\,=\,A^{\ast}
(dp(\xi_{0},\cdots ,\xi_{k},F,\overset{(r-k-1)}{\cdots },F))\\
=\, {\textstyle\sum}
(-1)^{i}p(\xi_{0},\cdots ,\xi_{i-1},\nabla^{A}\xi_{i},\xi_{i+1},\cdots ,\xi_{k}
,F^{A},\overset{(r-k-1)}{\cdots },F^{A})\, ,
\end{gather*}
and then
\[
d_{\mathcal{A}}\beta_{k}(\xi_{0},\cdots ,\xi_{k})\,=\,(r-k)d\left(
A^{\ast}( p(\xi_{0},\cdots ,\xi_{k},F,\overset{(r-k-1)}{\cdots
},F))\right)\, ,
\]
proving the proposition.
\end{proof}

\begin{remark}
In particular, for $k=r$ we have
\[
d_{\mathcal{A}}\beta_{r}^{p}\,=\, 0\, .
\]
\end{remark}

A connection on $P$ gives a $1$-forms in $P$ with values in
the Lie algebra $\mathfrak{g}$. Therefore,
we can equip $\mathcal{A}$ with the structure of a Fr\'{e}chet manifold. Let
$\mathcal{F}\,\subset\,\mathcal{A}$ be the closed subspace defined by the flat
connections on $P$. This subspace may be empty, but we are interested in the case
where $\mathcal{F}\,\neq\,\varnothing$. Take a compact oriented manifold
$S$, with $\dim S\,=\,k-1$, together with a smooth map
\[
f\,:\,S\,\longrightarrow\, \mathcal{F}\, .
\]
We assume that there is another compact oriented manifold $T$ of dimension $k$
such that
\[
\partial T\,=\, S\, ,
\]
and the inclusion $S\, \hookrightarrow\, T$ is compatible with the orientations
of $S$ and $T$. Since $\mathcal{A}$ is contractible, there is a smooth extension
\[
\overline{f}\,:\,T\,\longrightarrow\, \mathcal{A}\, ,\ \ \overline{f}\vert_{S}\,=\,f\, .
\]

We now define the main object of the work: given $S$, $f$, $T$ and $\overline{f}$
as above, consider the integral of the pull-back by $\overline{f}$ of the form
$\beta_{k}^{p}\,\in\, \Omega^{k}(\mathcal{A},\,\Omega^{2r-k}(M))$
\begin{equation}
\Lambda_{k}^{p}(f)\,:=\,\int_{T}\overline{f}^{\ast}\beta_{k}^{p}\,\in\,\Omega^{2r-k}
(M)\, .\label{def}
\end{equation}

There is an alternative way of defining $\Lambda_{k}^{p}(f)$ which
is close to the one used in \cite{Gu}. Consider the extended
principal bundle $T\times P\,\longrightarrow\, T\times M$, endowed
with the canonical connection $\widehat{A}$ defined as follows: The horizontal
lift of tangent vectors to $T$ is trivial, whereas the horizontal
lift of $v\,\in\, T_{x}M$ at a point $(z,\, x)\,\in\, T\times M$
is the one given by the connection
$\overline{f}(z)\,\in\,\mathcal{A}$. Let $\widehat{F}$ be the
curvature form of this connection, and let
$p(\widehat{F})\,\in\,\Omega^{2r}(T\times M)$ be the associated
characteristic form defined by the invariant polynomial $p$.

\begin{proposition}\label{alternative}
The form $\Lambda_{k}^{p}(f)\,\in\,\Omega^{2r-k}(M)$ is the
partial integration on $T\times M$
\begin{equation}\label{alternativeeq}
\Lambda_{k}^{p}(f)\,=\, \int_{T}p(\widehat{F})\, .
\end{equation}
\end{proposition}

\begin{proof}
Given coordinates $(z^{1}\, ,\cdots \, ,z^{k})$ on an open subset of $T$, it is easy
to check that
\begin{align*}
\widehat{F}(\partial/\partial z^{i},\cdot)|_{TM} & \,=\,\overline{f}^{\ast}(\partial
/\partial z^{i})\,\in\,\Omega^{1}(M,\,\mathrm{ad}P)\, ,\\
\widehat{F}(\partial/\partial z^{i},\partial/\partial z^{j}) & \,=\,0\, .
\end{align*}
Then
\begin{align*}
\int_{T}p(\widehat{F}) & \,=\,\int_{T}(i_{\partial/\partial z^{1}}\cdots i_{\partial
/\partial z^{k}}p(\widehat{F}))dz^{1}\cdots dz^{k}\\
& \,=\,N_{r,k}\int p(\overline{f}^{\ast}(\partial/\partial z^{1}),\cdots ,\overline{f}^{\ast
}(\partial/\partial z^{k}),F^{f(z)},\overset{(r-k)}{\cdots},F^{f(z)}
)dz^{1}\cdots dz^{k}\, ,
\end{align*}
and, taking into account Lemma \ref{lem2.1}, this integral is
precisely $\int_{T}\overline{f}^{\ast}\beta_{k}^{p}$.
\end{proof}

\begin{remark}
Furthermore, in the same vein of Proposition \ref{alternative}, we can provide a third approach to the form $\Lambda ^p _k (f)$. Let $A_0$ be any fixed connection on $P\to M$. We use the same notation $A_0$ for the pull-back connection on $T\times P\to T\times M$ defined by the projection $T\times M\to M$. The transgression formula provides a form $\vartheta (\widehat{A},A_0)\in \Omega ^{2r-1}(T\times M)$ such that
\[
p(\widehat{F})-p(F_0)=d\vartheta (\widehat{A},A_0).
\]
Then, from formula \eqref{alternativeeq} we have
\[
\Lambda ^p _k (f) = \int _Tp(F_0) +\int _Td\vartheta (\widehat{A},A_0) = \int _S \vartheta (\widehat{A},A_0),
\]
where the integral of $p(F_0)$ along $T$ vanishes as it is the pull-back of the characteristic form of the bundle $P\to M$ by the projection $T\times M\to M$. This definition of $\Lambda ^p_k (f)$ may be used for some of the proofs of the following. However, note that the form $\vartheta (\widehat{A},A_0)$ is not gauge invariant as it depends on the chosen connection $A_0$.

\end{remark}

\begin{lemma}\label{lem2}
The differential form $\Lambda_{k}^{p}(f)\,\in\,\Omega^{2r-k}(M)$,
$k\leq r$, constructed in \eqref{def} is closed.
\end{lemma}

\begin{proof}
The case $k=0$ being trivial, we assume $k>0$. As $d$ is a linear
operator, we have
\begin{align*}
d\Lambda_{k}^{p}(f) &
\,=\,d\int_{T}\overline{f}^{\ast}\beta_{k}^{p}\,=\,\int_{T}
\overline{f}^{\ast}(d\circ\beta_{k}^{p})\\
&
\,=\,\frac{1}{r-k+1}\int_{T}\overline{f}^{\ast}(d_{\mathcal{A}}\beta_{k-1}^{p})\,=\,\frac{1}{r-k+1}\int_{S}\overline{f}^{\ast}
\beta_{k-1}^{p}\, ,
\end{align*}
where we have taken into account Proposition \ref{Prop1}. The last
integral is zero as $\beta_{k-1}^{p}$ vanishes along
$\mathcal{F}$.
\end{proof}

Lemma \ref{lem2} produces a map
\begin{align*}
\Lambda_{k}^{p}\,\colon\,\mathrm{Map}(S,F) & \longrightarrow\, H^{2r-k}(M,\, \mathbb{R})\\
f & \,\longmapsto\, [\Lambda_{k}^{p}(f)]
\end{align*}
which is also denoted by $\Lambda_{k}^{p}$ for notational convenience.
Furthermore, from the definition of $\beta ^p _k$ given in
Proposition \ref{alternative} we have the following:

\begin{lemma}\label{lem3}
The cohomology class $\Lambda^p_k$ is the cap product
$$
(p (q^*_M P))\cap [T]\, \in\, H^{2r-k}(M,\, \mathbb{R})\, ,
$$
where $q_M\, :\, T\times M\, \longrightarrow\, M$ is the natural
projection, and $p (q^*_M P)\, \in\, H^{2r}(T\times M,\,
\mathbb{R})$ is the characteristic class of the pull-back
principal bundle $q^*_M P\to T\times M $, defined by the invariant
polynomial $p$.
\end{lemma}

Although the construction of $\Lambda_{k}^{p}(f)$ depends a priori
on the choice of $T$ and the extension $\overline{f}$, the
following proposition shows that $\Lambda_{k}^{p}(f)$ is actually
independent of them.

\begin{proposition}
\label{cob}The element $\Lambda_{k}^{p}(f)\,\in\, H^{2r-k}(M,\,
\mathbb{R})$ does not depend on either the choice of $T$ or the
choice of the extension $\overline{f}
\,:\,T\,\longrightarrow\,\mathcal{A}$.
\end{proposition}

\begin{proof}
Let $\overline{f}_{1}\,:\,T_{1}\,\longrightarrow\,\mathcal{A}$ and
$\overline{f}_{2} \,:\,T_{2}\,\longrightarrow\,\mathcal{A}$ be two
extensions. We consider
$$
D :=\, T_{1}\cup_{S}(-T_{2})\, .
$$
(the orientation of $T_2$ is reversed) and we glue the maps
$\overline{f}_{1}$ and $\overline{f}_{2}$ to a map $\widehat{f}
\,:\,D \,\longrightarrow\, \mathcal{A}$. The $(2r-k)$--form on $M$
\[
\theta\, :=\, \int_{D}\widehat{f}^{\ast}\beta_{k}^{p}=\int _{T_1}
\overline{f}_1^*\beta ^p_k-\int _{T_2}\overline{f}_2 \beta ^P_k
\]
is closed by setting $T$ in Lemma \ref{lem2} to be $D$, because
the boundary of $D$ is empty. As in Lemma \ref{lem3}, the
cohomology class $\widetilde{\theta}\, \in\, H^{2r-k}(M,\,
\mathbb{R})$ given by the form $\theta$ is the cap product $(p
(q^*_M P))\cap [D]\, \in\, H^{2r-k}(M,\, \mathbb{R})$. But
$$
p (q^*_M P)\,=\, q^*_M p (P)\, ,
$$
so $(p (q^*_M P))\cap [D]\,=\, 0$. Therefore, we have
$\widetilde{\theta}\,=\, 0$. The proposition follows from this.
\end{proof}

\begin{proposition}
\label{hoty} Fix $k\,<\,r$. Let $f_{1}\,:\,S\,\longrightarrow\,\mathcal{F}$ and $f_{2}
\,:\,S\,\longrightarrow\,\mathcal{F}$ be homologous maps in $\mathcal{F}$. Then
\[
\Lambda_{k}^{p}(f_{1})\,=\, \Lambda_{k}^{p}(f_{2})
\]
in $H^{2r-k}(M,\,\mathbb{R})$.
\end{proposition}

\begin{proof}
Let $D\subset \mathcal{F}$ be a $k$-dimensional chain such that $\partial D =f_1(S)\cup (-f_2(S))$. The (oriented) set $D\cup (-\overline{f}_1(T))\cup \overline{f}_2(T)$ is a cycle, and hence it is a border of a $(k+1)$-chain $C$ in the affine space $\mathcal{A}$. Then we have
\[
\int _C d_{\mathcal{A}} \beta ^p _k = \int _D \beta ^p_k - \int _T \overline{f}_1 ^* \beta ^p _k + \int _T \overline{f}_2 ^* \beta p _k
= -\Lambda ^p _k (f_1) + \Lambda ^p _k (f_2)
\]
where the integral along $D$ vanishes as $\beta ^p _k$ restricts to $0$ on $\mathcal{F}$ if $k<r$. Hence, $\Lambda ^p _k (f_1)$ and $\Lambda ^p_k (f_2)$ define the same element in $H^{2r-k}(M,\, \mathbb{R})$ because using Proposition \ref{Prop1},
\[
\int _C d_\mathcal{A}\beta ^p _k = (r-k) \int _C d\circ \beta^p _{k+1} = (r-k)d\circ \int_C \beta ^p _{k+1} .
\]
\end{proof}

\subsection{Gauge invariance and the moduli space}

Let $C(S,\mathcal{F})$ denote the space of all smooth maps from
$S$ to $\mathcal{F}$. Recall that a gauge transformations $\Phi
\in \mathrm{Gau}P$ induces a transformation $\Phi _\mathcal{A}:
\mathcal{A}\to \mathcal{A}$ in the space of principal connections.
In particular, this action leaves the subset $\mathcal{F}$ of flat
connection invariant.

\begin{proposition}\label{globalgauge}
The mapping
\[
\Lambda^{p}_k\,:\,C(S,\mathcal{F})\,\longrightarrow\, H^{2r-k}(M,\,\mathbb{R})
\]
is invariant under the action of the gauge group in $\mathcal{F}$, that is,
for every $\Phi\in\mathrm{Gau}P$ and $f\,\in\, C(S,\mathcal{F})$,
\[
\Lambda^{p}_k(\Phi\cdot f)\,=\, \Lambda^{p}_k(f)\, ,
\]
where $\Phi \cdot f \in C(S,\mathcal{F})$ is defined as
$(\Phi\cdot f) (x) = \Phi _\mathcal{A} (f(x))$, $x\in S$.
\end{proposition}

\begin{proof}
It is a direct consequence of Proposition \ref{inv} and the fact
that $\mathcal{F}$ is invariant under $\Phi_\mathcal{A}$.
\end{proof}

The following proposition studies the behavior of $\Lambda^{p}_k$
under pointwise gauge transformations for invariant polynomials
defining integral characteristic classes. By pointwise we mean
that the gauge transformation is not fixed and may depend on the
$S$.

\begin{proposition}\label{Propg}
Take any $\bar{\Phi}\,:\,S\,\longrightarrow\,\mathrm{Gau}P$. For
any two maps $f_{1}\,:\,S\,\longrightarrow\, \mathcal{F}$ and
$f_{2}\,:\,S\,\longrightarrow \,\mathcal{F}$ with
\[
f_{2}(x)\, =\, (\bar{\Phi}(x))_\mathcal{A}(f_{1}(x))\, ,
\]
the following holds:
\[
\Lambda^{p}_k(f_{1})\,=\, \Lambda^{p}_k(f_{2})\ \ \ {\rm mod}\ \ H^{2r-k}
(M,\, \mathbb{Z})\, ,
\]
that is, the invariant polynomial $p$ produces characteristic
class for $G$-bundles with values in integral cohomology classes.
\end{proposition}

\begin{proof}
Take any $T$ with $\partial T\,=\,S$, and also take two extensions $\overline{f}_{1}
\,:\,T\,\longrightarrow\, \mathcal{A}$ and $\overline{f}_{2}\,:\,T
\,\longrightarrow\,\mathcal{A}$ of $f_{1}$
and $f_{2}$ respectively. We consider two copies of the principal $G$-bundle
\[
\overline{P}\,=\,T\times P\,\longrightarrow\, T\times M
\]
equipped with the following connections: both the connections are
trivial along the direction of $T$, and the connection
$\overline{f}_{1}(z)$ (respectively, $\overline{f}_{2}(z)$) on $P$
coincides the connection on $\{z\}\times
P\,\longrightarrow\,\{z\}\times M$, $z\,\in\, T$, in the first
(respectively, second) copy. The orientation of the second $T$ is
reversed so that we can glue along the boundary to get an oriented
manifold $D\,=\,T\cup_{S}T$. As for the $G$-bundles, we glue the
two copies of $\overline{P}$ using $\bar{\Phi}$, that is, $(z\,
,p)\,\in\, S\times P$ of the first copy is glued with $(z\,
,\bar{\Phi}(z)(p))\,\in \, S\times P$ of the second copy. The
connections of the two copies of $G$-bundles glue compatibly to
give a connection on the $G$-bundle over $D\times M$. We consider
any $(2r-k)$-dimensional cycle $C\,\subset\, M$, and the
restriction of the new bundle to $D\times C$. The integral of the
characteristic class gives an integer, but this integral is
precisely
\[
\Lambda^{p}_k(f_{1})(C)-\Lambda^{p}_k(f_{2})(C)\, ,
\]
and the proof is complete.
\end{proof}

\begin{notation}
{\rm Let $[C(S,\,\mathcal{F}/\mathrm{Gau}P)]_{\mathrm{lift}}$ be the homotopy classes of maps from
$S$ the quotient space $\mathcal{F}/\mathrm{Gau}P$ that lift to a map from $S$
to $\mathcal{F}$.}
\end{notation}

\begin{corollary}
There is a well defined map
\[
[ C(S,\,\mathcal{F}/\mathrm{Gau}P)]_{\mathrm{lift}}\,\longrightarrow\, H^{2r-k}(M,\,\mathbb{R}
)/H^{2r-k}(M,\,\mathbb{Z})\, .
\]
\end{corollary}

\begin{proof}
Given any $\varphi\,\in\, C(S,\,\mathcal{F}/\mathrm{Gau}P)$, take any map $f\,:\,
S\,\longrightarrow\, \mathcal{F}$ that projects to $\varphi$, and consider the
corresponding $\Lambda^{p}_k(f)$. In view of Proposition \ref{Propg} it does not
depend on the choice of $f$ up to integral cohomology.
\end{proof}

\begin{remark}
Taking into account Proposition \ref{hoty}, the presentation of the previous Corollary could be given for the set of homology classes of maps from $S$ to $\mathcal{F}/\mathrm{Gau}P$ that lift to a map from $S$ to $\mathcal{F}$. We'd rather follow the homotopy formulation instead.
 \end{remark}

\subsection{Application to homotopy groups\label{hom}}

Here, we set $k\,\geq\,2$, $S\,=\,S^{k-1}$ (the $(k-1)$-sphere)
and $T\,=\,\overline{B}^{k}\,\subset\,\mathbb{R}^{k}$ (the unit
closed ball). We also set $2r-k\,\leq\, n$ and $k\,<\,r$ so that
Proposition \ref{hoty} applies. Homotopy classes of maps
\[
f\,:\,S^{k-1}\,\longrightarrow\,\mathcal{F}
\]
are elements of the homotopy group $\pi_{k-1}(\mathcal{F})$. Given an
invariant polynomial $p$ of degree $r$, we thus have a map
\[
\Lambda^{p}\,:\,\pi_{k-1}(\mathcal{F})\,\longrightarrow\, H^{2r-k}(M,\mathbb{R})
\]
given by
\[
\lbrack f]\,\longmapsto\,
\int_{\overline{B}^{k}}\overline{f}^{\ast}\beta_{k}^{p}\, ,
\]
where $\overline{f}\,:\,\overline{B}^{k}\,\longrightarrow\,\mathcal{A}$ is any extension
of $f$.

The following is straight-forward to check.

\begin{proposition}
The above map $\Lambda^{p}\,:\,\pi_{k-1}(\mathcal{F})\,\longrightarrow H^{2r-k}\,
(M,\mathbb{R})$ is a homomorphism of groups.
\end{proposition}

\section{Invariants in Dolbeault cohomology}

In this section we assume that $M$ is a compact complex manifold and $G$ is a
complex Lie group. Let
$$\rho^{r,s}\,:\,\Omega^{m}(M)\,\longrightarrow\,\Omega^{r,s}(M)\, ,\ \ r+s\,=\,m$$
be the projections associated to the decomposition
\[
\Omega^{m}(M)\,=\, \bigoplus_{r+s=m}
\Omega^{r,s}(M)
\]
of complex $m$-forms on $M$. Given an invariant polynomial
$p\,\in\, S^{r}(\mathfrak{g}^{\ast})^{G}$
of degree $r$ and an integer $0\,\leq\, k\,\leq \,r$, we define
\[
\widetilde{\beta}_{k}^{p}\,\in\,\Omega^{k}(\mathcal{A},\,\Omega^{0,2r-k}(M))
\]
as the $(0,2r-k)$-component of the form $\beta_{k}^{p}$ constructed in (\ref{beta}),
more precisely,
\[
\widetilde{\beta}_{k}^{p}\,=\,\rho^{0,2r-k}\circ\beta_{k}^{p}\, .
\]

\begin{proposition}
\label{PropD1}
The following holds:
\[
d_{\mathcal{A}}\widetilde{\beta}_{k}^{p}\,=\,(r-k)\cdot\overline{\partial}\circ\widetilde{\beta}
_{k+1}^{p}\, ,
\]
where
\[
d_{\mathcal{A}}\,:\,\Omega^{k}(\mathcal{A},\Omega^{2r-k}(M))\,\longrightarrow\,
\Omega^{k+1}(\mathcal{A},\Omega^{2r-k}(M))
\]
is the differential of forms on $\mathcal{A}$ taking values in the vector
space $\Omega^{2r-k}(M)$, and
\[
\overline{\partial}\,:\,\Omega^{0,2r-k-1}(M)\,\longrightarrow\,\Omega^{0,2r-k}(M)
\]
is the Dolbeault differential on $M$.
\end{proposition}

\begin{proof}
As $\rho^{0,2r-k}$ is a fiberwise linear,
\[
d_{\mathcal{A}}(\rho^{0,2r-k}\circ\beta_{k}^{p})\,=\,\rho^{0,2r-k}\circ(d_{\mathcal{A}}
\beta_{k}^{p})\, .
\]
Then, from Proposition \ref{Prop1},
\[
d_{\mathcal{A}}\widetilde{\beta}_{k}^{p}\,=\,(r-k)\rho^{0,2r-k}\circ(d\mathbf{\circ}
\beta_{k+1}^{p})\,=\,(r-k)\overline{\partial}\circ(\rho^{0,2r-k-1}\circ\beta_{k+1}^{p}
)\,=\,(r-k)\overline{\partial}\circ\widetilde{\beta}_{k+1}^{p}
\]
completing the proof.
\end{proof}

According to the Chern-Weil construction of characteristic classes, for any $A\,\in\,
\mathcal{A}$ the form $\widetilde{\beta}_{0}^{p}(A)\,\in\, C(\mathcal{A},\,
\Omega^{0,2r}(M))$ is the $(0,r)$-part of the characteristic polynomial $p(F^{A})$.

Below we assume that $k\,>\,0$.

Let $\mathcal{F}^{0,2}$ denote the subset of $\mathcal{A}$ consisting of all
connections $A$ satisfying the condition that the $(0,2)$-part of the
curvature of $A$ vanishes, that is
\[
\mathcal{F}^{0,2}\, =\, \{A\,\in\,\mathcal{A}\,\mid\, \rho^{0,2}(F^{A})\,=\,0\}.
\]
We note that any $A\, \in\, \mathcal{F}^{0,2}$ produces a
complex structure on the principal $G$-bundle $P$. Let $S$ be a $(k-1)$-dimensional
compact oriented manifold, and let
\[
f\,:\,S\,\longrightarrow\, \mathcal{F}^{0,2}
\]
be a smooth map.

\begin{lemma}
\label{lem02}
Take any $k\, <\, r$.
The form $\widetilde{\beta}_{k}^{p}\,\in\,\Omega^{k}(\mathcal{A},\,\Omega^{2r-k}(M))$
vanishes along $\mathcal{F}^{0,2}$.
\end{lemma}

\begin{proof}
We have
\[
\widetilde{\beta}_{k}^{p}(A)(\xi_{1},\cdots ,\xi_{k})\,=\,\rho^{0,2r-k}(p(\xi_{1},\cdots ,\xi
_{k},F^{A},\overset{(r-k)}{\cdots},F^{A})),
\]
for $\xi_{1},\cdots ,\xi_{k}\in T_{A}\mathcal{A}$. Since $r-k\, >\,0$, and
$\rho^{0,2}(F^{A})\,=\,0$, the top $(0,2r-k)$-part of $p(\xi_{1},\cdots ,\xi_{k}
,F^{A},\overset{(r-k)}{\cdots },F^{A})$ is zero.
\end{proof}

Assume that $S\,=\,\partial T$, where $T$ as before is a $k$-dimensional oriented
manifold such that the inclusion $S\, \hookrightarrow\, T$ is compatible with
the orientations. Choose an extension
$\overline{f}\,:\,T\,\longrightarrow\,\mathcal{A}$. Let
$\widetilde{\Lambda}_{k}^{p}(f)\,\in\,\Omega^{0,2r-k}(M)$ be the element defined by
\[
\widetilde{\Lambda}_{k}^{p}(f)\,=\,\int_{T}\overline{f}^{\ast}\widetilde{\beta}_{k}^{p}
\,\in\,\Omega^{0,2r-k}(M)\, .
\]

\begin{proposition}
The form $\widetilde{\Lambda}_{k}^{p}(f)\in\Omega^{0,2r-k}(M)$ is
$\overline{\partial}$-closed.
\end{proposition}

\begin{proof}
From Proposition \ref{PropD1}, we have
\begin{align*}
\overline{\partial}\widetilde{\Lambda}_{k}^{p}(f) &
\,=\,\overline{\partial}\int_{T}\overline{f}^{\ast}\widetilde{\beta}_{k}^{p}
\,=\,\int_{T}\overline{f}^{\ast}(\overline{\partial}
\circ\widetilde{\beta}_{k}^{p})\\
&
=\,\int_{T}\overline{f}^{\ast}d_{\mathcal{A}}\widetilde{\beta}_{k-1}^{p}\,=\,\int_{S}f^{\ast
}\widetilde{\beta}_{k-1}^{p}\, .
\end{align*}
This last integral vanished by virtue of Lemma \ref{lem02}.
\end{proof}

Therefore, we have an element
\[
\widetilde{\Lambda}_{k}^{p}(f)\,\in\, H^{0,2r-k}(M)
\]
of the Dolbeault cohomology.

The following three Propositions can be proved by adapting the proofs of
Propositions \ref{cob}, \ref{hoty} and \ref{globalgauge}
respectively; this in particular involves substituting the de Rham differential $d$
by the Dolbeault differential $\overline{\partial}$ and taking into account
Lemma \ref{lem02}.

\begin{proposition}
Take any $1\, <\, k\, <\, r$.
The Dolbeault cohomology class $\widetilde{\Lambda}_{k}^{p}(f)$ does not depend
on either the choice of $T$ or the choice of the extension
$\overline{f}$.
\end{proposition}

\begin{proposition}
The Dolbeault class $\widetilde{\Lambda}_{k}^{p}(f)$ is gauge invariant, that is,
given any $\Phi\,\in\,\mathrm{Gau}P$,
\[
\widetilde{\Lambda}_{k}^{p}(\Phi\cdot f)\,=\,\widetilde{\Lambda}_{k}^{p}(f)\, ,
\]
where $\Phi$ acts on connection in the natural way.
\end{proposition}

\begin{proposition}
Take any $k\,<\,r$. If $f_{1}\,:\,S\,\longrightarrow\,\mathcal{F}^{0,2}$ and $f_{2}\,:\,S\,\longrightarrow\, \mathcal{F}^{0,2}$ are homologous in $\mathcal{F}^{0,2}$, then
\[
\widetilde{\Lambda}_{k}^{p}(f_{1})\,=\,\widetilde{\Lambda}_{k}^{p}(f_{2})
\]
in $H^{0,2r-k}(M)$.
\end{proposition}

As in Section \ref{hom}, for any $1\,<\,k\,<\,r$, setting $S\,=\,S^{k-1}$ and
$T\,=\, \overline{B}^{k}$, a map from the homotopy
groups to the Dolbeault cohomology
\[
\widetilde{\Lambda}_{k}^{p}\,:\, \pi_{k-1}(\mathcal{F}^{0,2})\,\longrightarrow\,
H^{0,2r-k}(M)
\]
is obtained.

Define
\[
H^{0,m}(M,\,\mathbb{Z})\,=\, \{\rho^{0,m}(\omega)\,\mid\, \omega\,\in\,
H^{m}(M,\, \mathbb{Z})\}\, .
\]

\begin{proposition}
Take two maps $\Phi\,:\,S\,\longrightarrow\,\mathrm{Gau}P$ and $f_{1}\,:\,
S\,\longrightarrow\,\mathcal{F}$, and define
$$
f_{2}\,:\,S\,\longrightarrow\,\mathcal{F}\, ,\ \ f_{2}(x)\,=\,\Phi(x)\cdot f_{1}(x)\, .
$$
Then
\[
\widetilde{\Lambda}_{k}^{p}(f_{1})\,=\,\widetilde{\Lambda}_{k}^{p}(f_{2})\ \
\operatorname{mod}\ H^{0,2r-k}(M,\, \mathbb{Z})\, .
\]
\end{proposition}

\begin{proof}
As in the proof of Proposition \ref{Propg}, for any
$(2r-k)$-cycle $C$, we have
\[
\int_{C}\int_{T}\overline{f}_{1}^{\ast}\beta_{k}^{p}-\int_{C}\int_{T}\overline{f}
_{2}^{\ast}\beta_{k}^{p}\,\in\,\mathbb{Z}\, ,
\]
and hence
\[
\int_{T}\overline{f}_{1}^{\ast}\beta_{k}^{p}-\int_{T}\overline{f}_{2}^{\ast}\beta
_{k}^{p}\,\in\, H^{2r-k}(M,\mathbb{Z})\, .
\]
Then
\[
\rho^{0,2r-k}\left(
\int_{T}\overline{f}_{1}^{\ast}\beta_{k}^{p}-\int_{T}
\overline{f}_{2}^{\ast}\beta_{k}^{p}\right)\,=\,\Lambda_{k}^{p}(f_{1})-
\Lambda_{k}^{p}(f_{2})\in H^{0,2r-k}(M,\, \mathbb{Z})
\]
completing the proof.
\end{proof}

\begin{corollary}
There is a well defined homomorphism
\[
[C(S,\mathcal{F}^{0,2}/\mathrm{Gau}P)]_{\mathrm{lift}}\,\longrightarrow\, H^{0,2r-k}(M)/
H^{0,2r-k}(M,\,\mathbb{Z})\, ,
\]
defined on the set of homotopy classes of maps from $S$ to $\mathcal{F}^{0,2}/\mathrm{Gau}P$ that can be lifted to $\mathcal{F}^{0,2}$.
\end{corollary}

\section{Examples}

\subsection{Surfaces}

Recall that the form $\beta_{k}^{p}$ is not necessarily trivial
for
\[
2r-k\,\leq\, n\, ,\ \ k\,\leq\, r\, .
\]
Furthermore, the homomorphism $\Lambda^{p}\,:\,\pi_{k-1}(\mathcal{F)}\,\longrightarrow\,
H^{2r-k}(M)$ defined above only makes sense for $2\,\leq\, k\,<\,r$. If $M$ is a
surface ($n\,=\,2$), we have the following:
\begin{itemize}
\item Assume that $r\,=\,1$ and $k\,=\,0$. Then $\beta_{k}^{p}\,\in\,
C^{\infty}(\mathcal{A},\, \Omega^{2}(M))$ is the Chern-Weil characteristic form map.

\item Assume that $r\,=\,1$ and $k\,=\,1$. Then $\beta_{1}^{p}\,\in\,
\Omega^{1}(\mathcal{A},\, \Omega^{1}(M))$ does not depend on $A\,\in\,\mathcal{A}$
\[
\beta_{1}^{p}(\xi_{1})\,=\,p(\xi_{1})\, .
\]
We do not have homotopy invariance as $k\,=\,r$ so Proposition \ref{hoty} does
not apply.

\item Assume that $r\,=\,2$ and $k\,=\,2$. Then $\beta_{2}^{p}\,\in\,\Omega^{2}(\mathcal{A},
\, \Omega^{2}(M))$ does not depend on $A\,\in\,\mathcal{A}$.
Again, as $k=r$, we do not have homotopy invariance.

In any case, for $G\,=\,{\rm U}(m)$ and
$p(X)\,=\,\mathrm{tr}(X^{2})$, integrating $\beta_{2}^{p}$ along
$M$ gives the canonical symplectic form on $\mathcal{A}$ (see
\cite{AB}) given by
\[
\omega(\xi_{1},\xi_{2})\,=\,\int_{M}\mathrm{tr}(\xi_{1}\wedge\xi_{2})\, .
\]
\end{itemize}

\subsection{Case of $k=1$}

We consider a set of two points $S\,=\,\{0\, ,1\}$ and $T\,=\,[0\,
,1]$. A map $f\,:\,S\,\longrightarrow\,\mathcal{F}$ just gives two
flat connections whereas the extension $\overline{f}\,:\,[0\,
,1]\,\longrightarrow\,\mathcal{A}$ is a path connecting them. The
invariant $\Lambda^{p}(f)\,\in\, H^{2r-1}(M,\,\mathbb{R})$ defined
as
\[
\Lambda^{p}([f])\,=\, \int_{0}^{1}\overline{f}^{\ast}\beta_{1}^{p}\, ,
\]
is independent of the extension $\overline{f}$ and the homotopy
class of $f$. If the two points $A_{1}\, ,A_{2}\,\in\,\mathcal{F}$ are in
the same arc-connected component of $\mathcal{F}$, then the
extension $\overline{f}$ can be defined entirely contained in
$\mathcal{F}$, so that $\Lambda^{p}([f])\,=\,0$. The map
$\Lambda^{p}([f])$ thus gives information of points in different
connected
components of $\mathcal{F}$. Fixing $A_{1}$, we can define
\begin{align*}
\pi_{0}(\mathcal{F}) & \,\longrightarrow\, H^{2r-1}(M,\,\mathbb{R})\\
K & \,\longmapsto\, \Lambda([f])\, ,
\end{align*}
where $f\,:\,\{0\, ,1\}\,\longrightarrow\,\mathcal{F}$, $f(0)\,=\,A_{1}$,
$f(1)\,\in \,K$, with $K$ being an arc-connected component of $\mathcal{F}$.

We develop the definition of $\Lambda^{p}([f])$. Given $A_{1}$ and $A_{2}$,
the extension $\overline{f}$ can be taken as the line
\[
\overline{f}(t)\,=\,(1-t)A_{1}+tA_{2}
\]
so that
\[
\Lambda^{p}([f])\,
=\,\int_{0}^{1}p(A_{2}-A_{1},F_{t},\overset{(r-1)}{\cdots},F_{t})dt\, ,
\]
where $F_{t}$ is the curvature of $A_{t}\,=\,(1-t)A_{1}+tA_{2}$.
This is precisely the expression of the transgression formula for
$p$, $A_1$ and $A_2$. In fact, writing $\xi\,=\,A_{2}-A_{1}$, we
have
\[
F_{t}\,=\,F_{A_{1}}+t\nabla^{A_{1}}\xi+t^{2}[\xi,\xi]\, ,
\]
and as $F^{A_{1}}\,=\,0$ and $F^{A_{1}+\xi}\,=\,\nabla^{A_{1}}\xi+[\xi,\xi]\,=\,0$, we
have
\[
F_{t}\,=\,(t^{2}-t)[\xi,\xi]\, .
\]
Then the invariant $\Lambda^{p}([f])$ is
\[
\Lambda^{p}([f])\,=\,p(\xi,[\xi,\xi],\overset{(r-1)}{\cdots},[\xi,\xi])
\int_{0}^{1}(1-t^{2})^{r-1}dt\, .
\]
For Abelian gauge theories, this expression is trivial, but when $G$
is non-Abelian,
we get a non-vanishing invariant. For example, consider $P\,=\,M\times {\rm SU}(2)$,
and let $p$ be the inner product on $\mathfrak{su}(2)\,\simeq\,\mathbb{R}^{3}$.
In this case,
\[
\Lambda^{p}([f])\,=\,-\frac{2}{3}\langle\xi,[\xi,\xi]\rangle\,=\,-\frac{2}{3}\det
(\xi,\xi,\xi)\in H^{3}(M,\mathbb{R})\, ,
\]
$\xi\,\in\,\Omega^{1}(M,\,\mathfrak{su}(2))\,=\,\Omega^{1}(M,\,\mathbb{R}^{3})$.
If we put $\xi\,=\,\xi^{i}B_{i}$ for an orthonormal basis (for
instance, the Pauli matrices) $B_{1} ,B_{2},B_{3}$, then
\[
\Lambda^{p}([f])\,=\,-\frac{2}{3}\xi_{1}\wedge\xi_{2}\wedge\xi_{3}\, .
\]
As $d\xi+[\xi,\xi]\,=\,0$, this is equivalent to
\[
\Lambda^{p}([f])\,=\,-\frac{2}{3}\xi_{1}\wedge d\xi_{1}\, ,
\]
which is a Chern-Simons type formula.

\subsection{Case of $k=2$}

For $r>2$, we then have morphisms
\[
\Lambda_{2}^{p}\,:\,\pi_{1}(\mathcal{F})\,\longrightarrow\, H^{2r-2}(M,\mathbb{R}),
\]
and
\[
\Lambda_{2}^{p}\,:\,[C(S^1,\mathcal{F}/\mathrm{Gau}P)]_{\mathrm{lift}}
\,\longrightarrow\,
H^{2r-2}(M,\mathbb{R})/H^{2r-2}(M,\mathbb{Z})
\]
which make sense if $2r-2\,\leq\,\dim M$. The first relevant case
is $\dim M=4$ and $r=3$. For example, when $G\,=\,{\rm U}(2)$, the
set of invariant polynomials is generated by the trace $p_{1}$ and
the determinant $p_{2}$. For $n=4, r=3$, we can consider either
$p\,=\,p_{1}^{3}$ or $p\,=\,p_{2}p_{1}$. We then have
\begin{align*}
\Lambda_{2}^{p} & \,:\,\pi_{1}(\mathcal{F})\,\longrightarrow\,\mathbb{R}\, ,\\
\Lambda_{2}^{p} & \,:\,
[C(S^1,\mathcal{F}/\mathrm{Gau}P)]_{\mathrm{lift}}
\,\longrightarrow \,S^{1}\, .
\end{align*}
Given a loop $A(s)$ in $\mathcal{F}$, we can choose a simple extension as
$\overline{f}\,:\,[0\, ,1]\times\lbrack0,1]\,\longrightarrow\,\mathcal{A}$
\[
\overline{f}(s,t)\,=\,(1-t)A_{0}+tA(s)\, ,
\]
where $A_{0}\,=\,A(0)$. Then
\begin{align*}
\Lambda_{2}^{p}(f) & =\int_{0}^{1}\int_{0}^{1}p(\partial_{t}\overline{f}
,\partial_{s}\overline{f},F^{\overline{f}(s,t)},\overset{(r-2)}{\cdots},F^{\overline{f}(s,t)})\\
& =\int_{0}^{1}\int_{0}^{1}p(A(s)-A_{0},t\dot{A}(s),F^{\overline{f}(s,t)}
,\overset{(r-2)}{\cdots},F^{\overline{f}(s,t)})dtds\, ,
\end{align*}
where $\dot{A}(s)\,=\,\partial_{s}A(s)$. As both $A_{0}$ and $A(s)$ are flat
connections, it is easy to see that
\[
F^{\overline{f}(s,t)}\,=\,F^{(1-t)A_{0}+tA(s)}\,=\,(t^{2}-1)[A(t)-A_{0},A(t)-A_{0}]\, .
\]
Then, setting $\xi(s)\,=\,A(s)-A_{0}$, we have
\begin{align*}
\Lambda_{2}^{p}(f) & =\int_{0}^{1}t(t^{2}-1)^{r-2}dt\int_{0}^{1}p(\xi
(s),\dot{\xi}(s),[\xi(s),\xi(s)],\overset{(r-2)}{\cdots},[\xi(s),\xi(s)])ds\\
& =\frac{(-1)^{r}}{2(r-1)}\int_{0}^{1}p(\xi(s),\dot{\xi}(s),[\xi
(s),\xi(s)],\overset{(r-2)}{\cdots},[\xi(s),\xi(s)])ds\\
& =\frac{1}{2(r-1)}\int_{0}^{1}p(\xi(s),\dot{\xi}(s),\nabla^{A_{0}}
\xi(s),\overset{(r-2)}{\cdots},\nabla^{A_{0}}\xi(s))ds\, ,
\end{align*}
where $\dot{\xi}$ stands for $d\xi/ds$; for the last step we have taken
into account that $F^{A_{0}+\xi(s)}\,=\,\nabla^{A_{0}}\xi(s)+[\xi(s),\xi(s)]\,=\,0$.
In particular, this expression has a simpler form when the bundle $P $ is
trivial, $A_{0}$ is chosen to be the trivial connection and connections are
seen as $1$-forms in $M$ taking values in $\mathfrak{g}$, that is
\begin{equation}\label{laas}
\Lambda_{2}^{p}(f)\,=\,\frac{1}{2(r-1)}\int_{0}^{1}p(A(s),\dot{A}
(s),[A(s),A(s)],\overset{(r-2)}{\cdots},[A(s),A(s)])ds\, .
\end{equation}
If $A(s)=(\Phi(s))_\mathcal{A}(A_0)$, for a loop $\Phi(s) \subset \mathrm{Gau}P$, $0\leq s \leq 1$, the definition of degree of the loop is
\[
\int _M \int _0 ^1 p(A(s),\dot{A}(s))ds,
\]
for a metric $p$ (that is, $r=2$) in $\mathfrak{g}$ defining integer cohomology class (see, for example, \cite{Sal}).
We can understand the expression \eqref{laas} as a higher order generalization of the degree. Of course, the integral class vanishes when considered in $[S^1,\mathcal{F}/\mathrm{Gau}P]_{\mathrm{lift}}$. The expression for other loops in $A(s)$ in $\mathcal{F}$ is to be analyzed.

\section*{Acknowledgements}

We thank the referee for very helpful comments.
The first author is supported by a J. C. Bose Fellowship.
The second author thanks the TATA Institufe of Fundamental Research for its hospitality during a visit in which a part of this work was developed. The second author was partially funded by MINECO (Spain) Project MTM2015-63612-P.

\end{document}